\documentclass[en, biber
, secnum]{cls/myamspaper}
\title{Compact right topological semigroups applied to operators}

\date{November 20, 2023}
\usepackage{stmaryrd}
\usepackage{tikz-cd}
\usepackage{todonotes}
\renewcommand{\K}{\mathcal{K}}
\newcommand{\sub}{\subseteq}
\newcommand{\rev}{E_\mathrm{rev}}
\newcommand{\aws}{E_\mathrm{aws}}
\newcommand{\per}{\sigma_{\mathrm{per}}}
\newcommand{\w}{{\sigma^*}}
\begin{document}

\begin{abstract}
  \setlength\parindent{0em}
  \setlength\parskip{0.3\baselineskip}

  In this paper we introduce a new decomposition of power-bounded operators, analogous to the Jacobs-deLeeuw-Glicksberg decomposition (see \newline\cite[Chap. 15.6]{Eisner2016}). This is done using so-called Köhler semigroups and the general theory of right topological compact semigroups. This leads to a new proof, under appropriate assumptions, of the cyclicity of the peripheral spectrum of positive operators on Banach lattices. 

\end{abstract}

\maketitle

\section{The Structure of Köhler Semigroups} 
Our central objects are Köhler semigroups of power-bounded operators $T$ on dual Banach spaces $E$, meaning that $E$ is, in addition to the norm topology, equipped with a weak-* topology on $E$ for which the operator $T$ is continuous. So there exists a Banach space $F$ and an operator $S$ on $F$ such that the dual space of $F$ is $E$ and the adjoint of $S$ is $T$.  Note that the weak-* topology might differ for different choices of the predual, nevertheless we will usually not mention the pre-dual space of $E$ explicitly.
% \vspace{-0.1cm}
\begin{definition}
    For any power-bounded weak-* continuous operator $T$ on some dual Banach space $E$ we define the \textit{Köhler semigroup} $\K(T)$ to be the closure 
    \[\K(T)=\overline{\{T^n:\,n\in\N\}}^{\w}\sub E^E\]
    with respect to pointwise weak-* convergence. 
\end{definition}
Köhler semigroups have first been considered by Witz \cite{Witz1964} and investigated in detail by Köhler \cite{kohler1994enveloping} (see also \cite{romanov2011weak,Kreidler2018CompactOS}) and are the operator-theoretic analogue to the Ellis semigroups in topological dynamics (see \cite{ellis1960semigroup}, or \cite{romanov2011weak,pym1990compact}). 

We now show that every such Köhler semigroup is (canonically) a factor of the Stone-{\v C}ech compactification $\beta\N$ of the natural numbers. The space $\beta\N$ has the universal property that every map from $\N$ to some compact Hausdorff space has a unique continuous extension to the compactification $\beta\N$ of $\N$. Further, it can be identified with all ultrafilters on $\N$ and forms a compact \textit{right topological} semigroup, i.e., for every $p\in\beta\N$ the map $s\mapsto s+p$ is continuous. Here $\N$ is identified with the set of all fixed ultrafilters and the (non-commutative!) addition $p+q$ of two ultrafilters $p,\,q\in\beta\N$ is the ultrafilter consisting of all sets $A\sub \N$ such that 
\[\{n\in\N: A-n\in q\}\in p.\]

For more on $\beta\N$ we refer to \cite{hindman2011algebra}. The next theorem is essentially \cite[Theorem 3.2]{kohler1994enveloping}.

\begin{theorem}\label{frombetan}
The Köhler semigroup $\K(T)$ (of a power-bounded operator $T$ with preadjoint) is a factor of $\beta\N$ via the continuous semigroup-epimorphism $p\mapsto\lim_{n\to p} T^n$ from $\beta\N$ to $\K(T)$, where $\lim_{n\to p}T^n$ denotes the limit of the sequence $T^n$ along the ultrafilter $p$.   
\end{theorem}
\begin{proof}
First note that the sequence $(T^n)_{n\in\N}$ is bounded and hence contained in some $B_r(E)^E$, where $B_r(E)$ denotes the closed ball in $E$ with radius $r$ centered at $0$. Since $B_r(E)$ is weak-*-compact, $B_r(E)^E$ is compact and thus $\{T^n:\,n\in\N\}$ is relatively compact with respect to the pointwise weak-* convergence. Thus for any ultrafilter $p\in\beta\N$ the sequence $(T^n)$ converges along $p$, i.e., $p\mapsto \lim_{n\to p}T^n$ is well-defined. 
For the continuity of $p\mapsto \lim_{n\to p}T^n$ take any open subset $O$ in $E^E$ and denote by $A\sub\N$ the set of indices $n$ such that $T^n\in O$. By definition of the limit the preimage of $O$ consists of all ultrafilters containing $A$. However this corresponds to the fixed filter $\Bar{A}\sub \beta\N$ and any fixed filter is a clopen set in $\beta\N$ (as its complement in $\beta\N$ is a fixed filter as well). Thus, by definition, the map $p\mapsto \lim_{n\to p}T^n$ is continuous.

For any (open) set $U$ the set of indices $A=\{n\in\N\colon T^n\in U\}$ is in $p+q$ precisely if the set $B\coloneqq \{n\in \N \colon T^n\lim_{k\to q}T^k\in U\}$ is in $p$, as by weak-*-continuity of $T$ we can rewrite $B$ as
\[\{n\in \N\colon  \lim_{k\to q} T^{n+k}\in U\}=\{n\in \N\colon \{k \colon T^{n+k}\in U\}\in q\}=\{n\in \N\colon A-n\in q\}\]
and by definition of $p+q$ this lies in $p$ exactly for $A\in p+q$. This means that both limits converge to the same point and thus $p\mapsto \lim_{n\to p}T^n$ is indeed a semigroup homomorphism.

Since $\K(T)$ is the closure of $\{T^n\colon n\in\N\}$, every point in $\K(T)$ is an ultrafilter limit of elements in $\{T^n\colon n\in\N\}$. This clearly corresponds to an ultrafilter in $\beta\N$, so the map is an epimorphism.
\end{proof}
The following example from \cite[Example 3.5]{kohler1994enveloping} shows that the space $\beta\N$ itself can also be seen as a Köhler semigroup, and the above epimorphism can be an isomorphism.
\begin{example}    The Köhler semigroup of the left shift $L\colon\,(a_n)_{n\in\N}\mapsto (a_{n+1})_{n\in\N}$ on the (dual) space $\ell^\infty(\N)\cong C(\beta\N)$ (with weak*-topology being induced by $\ell^1$) is isomorphic to the Stone-{\v C}ech compactification $\beta\N$ of the natural numbers.
\end{example}
\begin{proof} 
Based on Theorem \ref{frombetan} it suffices to show the injectivity of $p\mapsto\lim_{n\to p}L^n$. Thus consider different ultrafilters $p\ne q\in \beta\N$. Then there exists a subset $A\sub \N$ with $A\in p$, but $A\notin q$, so $A^c\in q$. Consider the characteristic sequence $\1_{A+1}$ of the set $A+1$. Then \[\{n\in \N:\, (L^n \1_{A+1})(1)=1\}=A,\] and thus 
$(\lim_{n\to p}L^n\1_{A+1})(1)=1$ but $(\lim_{n\to q}L^n\1_{A+1})(1)=0$,
as the weak-* topology on bounded subsets in $\ell^\infty$ is the topology of pointwise convergence on $\N$. Hence $\lim_{n\to p} L^n\ne \lim_{n\to q}L^n$, so $\K(T)\cong \beta\N$.
\end{proof}

The structure of $\beta\N$ directly implies various properties of Köhler semigroups, see e.g. \cite[Lem. 2.2]{Kreidler2018CompactOS}, \cite[Prop. 2.3]{kohler1994enveloping}.
\begin{corollary}
    The Köhler semigroup $\K(T)$ of a power-bounded operator $T$ is a compact right topological semigroup such that $T$ commutes with every element of $\K(T)$. 
\end{corollary}

Note that $\beta\N$ is a highly non-commutative semigroup. The degree of non-commutativity is directly related with the continuity of the left-multiplication. Recall that the \textit{algebraical center} of a right topological semigroup consists of all elements commuting with every other element of the semigroup, while the \textit{topological center} consists of all elements $R$ for which the left-multiplication $S\mapsto RS$ is continuous. 

\begin{proposition}\label{centers}
The algebraical and topological centers of $\K(T)$ coincide and contain all operators having a preadjoint.
\end{proposition}
\begin{proof}
The proof follows essentially \cite[Thm 4.24]{hindman2011algebra}.
First note that all operators with preadjoint are in the topological center of $\K(T)$ as they are weak-* continuous. Each $T^n$ is in the algebraical center of $\K(T)$ since $\N$ is the algebraical center of $\beta\N$ (see \cite[Thm 6.10]{hindman2011algebra}). Clearly the algebraical center is contained in the topological center.

Let $R$ be in the topological center of $\K(T)$, and fix any $S\in \K(T)$. Since $\{T^n:n\in\N\}$ is dense in $\K(T)$, there exists a filter $s$ on $\N$ such that $S=\lim_{n\to s}T^n$. This implies
\[ RS=R\lim_{n\to s}T^n=\lim_{n\to s}RT^n=\lim_{n\to s}T^nR=SR.\]
~
\end{proof}

\begin{remark}
    The above proof actually shows  that if the algebraical center of a right topological semigroup is dense in the semigroup, then it coincides with the topological center. 
\end{remark}
Next, we see that well known monothetic semigroup compactifications as, e.g., the Jacobs-deLeeuw-Glicksberg compactification of \cite[Chap 15.6]{Eisner2016} are special cases of  Köhler semigroups. The statement is Proposition 2.5 in \cite{kohler1994enveloping}. 

\begin{proposition}\label{waskkomp}
If the operator $T\colon E\to E$ has relatively weakly/norm compact orbits, then $\K(T')$ consists of the adjoint operators of the closure of the semigroup $\{T^n:n\in\N\}$ in weak/norm operator topology. In particular it is a commutative semitopological semigroup.
\end{proposition}
\begin{proof}
   Let $T$ have relatively weakly compact orbits $\{T^nx:\,n\in\N\}$ for all $x\in E$. Denote the pointwise weak closure of $\{T^n:n\in\N\}$ by $\mathcal{S}$. Clearly the adjoint operators $\{R':\, R\in \mathcal{S}\}$ are contained in the Köhler semigroup $\K(T')$ on $E'$. Further, for any $R=\lim_{n\to r} T'^n\in \K(T')$ (with suitable $r\in\beta\N$) we have that the pointwise weak limit $H=\lim_{n\to r} T^n$ exists and 
   \[ \langle H'x, y\rangle =\lim_{n\to r}\langle T^nx,y\rangle =\lim_{n\to r}\langle x,T'^ny\rangle=\langle x,\lim_{n\to r}T'^ny\rangle=\langle x,Ry\rangle.\]
   Thus $H'=R$ and hence $R$ has a preadjoint in $\mathcal{S}$.

   If furthermore $T$ even has relatively norm-compact orbits, then the weak closure coincides with the norm closure, and the statement follows by the preceeding argument.    \end{proof}

One of the most important theorems about compact right-topological semigroups is that they always posess idempotents, and even minimal idempotents. This will lead to decompositions of our Banach space and the given operator. 

 \begin{definition} 
    For a semigroup $\mathcal{S}$ an idempotent $P=P^2\in \mathcal{S}$ is called a \textit{minimal idempotent} if it is minimal with respect to the order on all idempotents  in $\mathcal{S}$ given by $P\le Q$ if $P=PQ=QP$.
\end{definition}
Idempotents in $\beta\N$ are well understood. We recall the following results, which can be found, e.g., in \cite[Lem 5.19.1, Thm 5.12 and Thm 6.9]{hindman2011algebra}.  
\begin{lemma}\label{lemma:betaN}
The idempotents in $\beta\N$ are precisely the ultrafilters $p$ such that for each $A\in p$ the set $\{n\in\N:\, A-n\in p\}$ is in $p$. In particular the following holds.
\begin{itemize}
    \item Each idempotent in $\beta\N$ contains the sets $k\N$ for each $k\in\N$.
\item For a subset $A\sub \N$ there exists an idempotent $p\in\beta\N$ with $A\in p$ if and only if there exists an infinite sequence $(x_n)_{n\in\N}\sub \N$ such that all finite sums of elements $x_n$ are contained in $A$.
\item 
Further, $\beta\N$ contains $2^{2^{|\N|}}$ minimal left ideals, right ideals and idempotents.
\end{itemize}
\end{lemma}

\section{Decompositions Induced by Köhler Semigroups}
Again, let $T$ be any power bounded operator with preadjoint on a dual Banach space $E$. 
Now we are interested in decompositions of our Banach space into $T$-invariant subspaces obtained by idempotents (i.e. projections) in the Köhler semigroup $\K(T)$. 
\begin{proposition}\label{idemfr}
 Idempotents in $\K(T)$ exist and are given by the images of idempotents in $\beta\N$ under the canonical projection.
\end{proposition}
\begin{proof}
    For any idempotent ultrafilter $p\in\beta\N$ the image $\lim_{n\to p}T^n\in\K(T)$ under the canonical epimorphism is again an idempotent.
    
    If we take any idempotent $P\in\K(T)$, then the preimage of $\{P\}$ in $\beta\N$ is a closed subsemigroup. Thus, the preimage is a compact right topological semigroup and hence contains an idempotent by a well known lemma of Ellis \cite{ellis1958distal}.
\end{proof}

Combining the previous results we arrive at our main theorem.
\begin{theorem}\label{rechtstop-JdLG}
Let $T$ be a power-bounded operator with preadjoint on a dual Banach space $E$. Every idempotent $P$ in the Köhler semigroup $\K(T)$ yields a decomposition of $E$ into $E=\rev\oplus\aws$, where  $\rev\coloneqq\im{P}$, and $\aws\coloneqq\ker{P}$. This decomposition fulfills the following properties.
    \begin{enumerate}
        \item The spaces $\rev$ and $\aws$ are invariant under $T$ and also under all operators in $\K(T)$ with preadjoint.
        \item The restricted operator $T|_{\rev}$ is invertible. If $P$ is a minimal idempotent (they always exist!), every operator in $P\K(T)$ is invertible on $\rev$ with inverse in $P\K(T)$.
        \item All eigenvectors of $T$ corresponding to unimodular eigenvalues are contained in $\rev$. Each $x\in \rev$ is weak-* accumulation point of its own orbit $\{T^n x:\,n\in\N\}$.
        \item If $T^nx$ converges to $0$, then $x\in\aws$. Each $x\in\aws$ has $0$ as a weak-* accumulation point of its orbit $\{T^nx:\, n\in\N\}$.
    \end{enumerate}
\end{theorem}

\begin{proof}
Let $P$ be any projection in $\K(T)$ and denote by $p$ an idempotent ultrafilter on $\N$ corresponding to the projection $P$ (they exist by Lemma \ref{idemfr}).  

Part (a) follows directly from Proposition \ref{centers} since all such operators commute with $P$. 

For part (b) first note that for each ultrafilter $p\in\beta\N$ with $p\ne 1$ the set $p-1$ consisting of all sets $A-1\coloneqq \{t\in \N: t+1\in A\}$ for $A\in p$ is again an ultrafilter on $\N$. Further by definition the ultrafilter $1+(p-1)$ consist of all sets $A\sub \N$ such that $1\in \{ n\in \N\colon A-n\in (p-1)\}$, which is equivalent to $A-1\in p-1$ and thus to $A\in p$. This implies $p+1+(p-1)=(p-1)+p+1=p+p=p$ in $\beta\N$. Therefore the operator $T$ restricted to $\rev$, corresponding to $p+1\in \beta\N$, is invertible with inverse $P\lim_{n\to p-1}T^n$ in $\K(T)$ restricted to $\rev$.

 Minimal idempotents in compact right topological semigroups exist by a theorem of Ruppert \cite[Thm I.3.11]{berglund1989analysis}.
If we choose $P$ to be a minimal idempotent, then $P\K(T)P$ is a group (see \cite[Thm I.2.8]{berglund1989analysis}), thus for each $PQ\in P\K(T)$ there exist $PRP$ such that $(PQP)(PRP)=(PQ)(PR)P=P$. Hence, on $\rev$ we have $(PQ)(PR)=\mathrm{Id}$. 

Since the projection $P$ is a pointwise weak-* accumulation point of $\{T^n:\,n\in\N\}$, $Px$ is a weak-* accumulation point of $(T^nx)_{n\in\N}$. 
If $Tx=\lambda x$ for $|\lambda|=1$, the orbit $T^n x$ is contained in the weak-* compact subset $\{\alpha x\colon |\alpha|=1\}$. Thus $Px=\alpha x$ for some $\alpha$ with $|\alpha|=1$. Since $P$ is idempotent, this implies $Px=x$ and thereby part (c). The last assertion for $x\in \aws$ follows analogously.
\end{proof}
\begin{remark}
\begin{enumerate}[(i)]
    \item
    The space $\rev$ is called the reversible part of $E$, and $\aws$ the stable part of $E$ with respect to $P$. This corresponds to the terminology of the Jacobs-deLeeuw-Glicksberg decomposition (see \cite[Chap 15.6]{Eisner2016}). Note that $\rev$ and $\aws$ depend on the chosen idempotent. If we want to emphasise this dependency on the idempotent $P$, we write ${\rev}_P$ and ${\aws}_P$.
    \item Similar decompositions can be obtained for strongly continuous one-parameter semigroups (e.g. in \cite{budde2022application}).
    \item 
    If the preadjoint of $T$ has relatively weakly compact orbits, then there is a unique minimal idempotent in $\K(T)$  and the induced decomposition corresponds to the adjoint of the Jacobs-deLeeuw-Glicksberg decomposition. This follows by Proposition \ref{waskkomp} and the existence of a unique minimal idempotent in the preadjoint semigroup (see e.g. \cite[Thm 16.5]{Eisner2016}). 
    \end{enumerate}

\end{remark}
As $\beta\N$ has $2^{2^{|\N|}}$ minimal idempotents, but the classical Jacobs-deLeeuw-Glicksberg decomposition is unique, one might wonder whether two distinct idempotents might induce the same decomposition.
\begin{lemma}\label{lem:minidem}
The minimal left ideals in $\K(T)$ correspond to $\{{\aws}_P\colon P \text{ min. idem.}\}$, and the minimal right ideals correspond to $\{{\rev}_P\colon P \text{ min. idem.}\}$. Further, for each choice of a stable part in $\{{\aws}_P\colon P \text{ min. idem.}\}$ and a reversible part in $\{{\rev}_P\colon P \text{ min. idem.}\}$, there is precisely one idempotent $P\in \K(T)$ inducing this decomposition. 
\end{lemma}
\begin{proof}
All minimal idempotents lie in minimal left ideals, thus there is an obvious surjection from minimal left-ideals onto $\{{\aws}_P\colon P \text{ min. idem.}\}$. It remains to show that two minimal idempotents $P,\, Q$ induce the same space ${\aws}_P={\aws}_Q$ precisely if they are in the same minimal left ideal. If $\ker(P)=\ker(Q)$, then in particular $\ker(Q)\subseteq\ker(P)$ and hence $PQ=P$, so $\K(T)P\subseteq \K(T)Q$ which by minimality implies equality, hence both idempotents correspond to the same minimal left-ideal. For the other implication assume they belong to the same minimal left ideal. This implies that $Q\in \K(T)P$ and $P\in \K(T)Q$, i.e., $\ker(P)\sub \ker(Q)$ and $\ker(Q)\sub \ker(P)$, so ${\aws}_P={\aws}_Q$.

Similar reasoning can be made for $\rev$. For two minimal idempotents $P,\, Q$ the statement ${\rev}_P={\rev}_Q$ is equivalent to $PQ=Q$ and $QP=P$, which is equivalent to $P\K(T)=Q\K(T)$.

Finally for any minimal right ideal $R$ and minimal left ideal $L$ the intersection $L\cap R$ is a group, which has exactly one idempotent $e\in L\cap R$ such that $R\cap L=RL=e\K(T)e$ with $e\K(T)=R$ and $\K(T)e=L$ (see \cite[Thm 1.61]{hindman2011algebra}). This implies the last statement.
\end{proof}
In the general case this means that the decomposition of Theorem \ref{rechtstop-JdLG} is highly non-unique.
\begin{example}
Consider the left shift $L$ on the space $\ell^\infty$ and 
define the space of all almost periodic sequences in $\ell^\infty$ as the norm-closure of the periodic sequences 
    \[\mathcal{A}\coloneqq\overline{\{f\in\ell^\infty: f \text{ is periodic}\}}.\]
        For every decomposition $\ell^\infty=\rev\oplus\aws$ we have $\mathcal{A}\subsetneqq \rev$ and $c_0\subsetneqq\aws$. 
        Further there are $2^{2^{|\N|}}$ many different spaces $\aws$ and $\rev$.
\end{example}
    \begin{proof}
    Take any periodic sequence $f$ and any idempotent $P\in \K(T)$. Then there is a $k\in \N$ such that $L^{nk}f=f$ for all $n\in\N$. But as the set $k\N$ is contained in any idempotent ultrafilter (see Lemma \ref{lemma:betaN}), hence $f=\lim_{n\to P}L^nf$. Now $\mathcal{A}\sub \Im(P)={\rev}_P$ follows immediately by (norm)-continuity of $P$. 
     To see $\mathcal{A}\subsetneqq {\rev}_P$, note that by Lemma \ref{lem:minidem} different minimal right ideals induce spaces ${\rev}_P,\,{\rev}_Q$ which are not included in each other. Thus we conclude that for minimal idempotents $\mathcal{A}\subsetneqq {\rev}_P$, because there are $2^{2^{|\N|}}$ minimal right ideals in $\beta\N\simeq \K(T)$ (see Lemma \ref{lemma:betaN}). For arbitrary idempotents $Q'$ the statement $\mathcal{A}\subsetneqq {\rev}_{Q'}$ follows since there exists a minimal idempotent $P\le Q'$, so $\mathcal{A}\subsetneqq {\rev}_P\sub {\rev}_{Q'}$.

Since for any sequence $(x_m)$ in $c_0$ the sequence (of sequences) $(L^n(x_m)_{m\in\N})_{n\in\N}$ converges pointwise to $0$, we obtain $c_0\sub \aws$. The inclusion $c_0\subsetneqq \aws$ is obtained analogously.
\end{proof}

\section{Application to the Spectral Theory of Positive Operators}
We now use the technique developed above in the context of positive operators on Banach lattices, referring to \cite{schaefer1974Banachlattices} for terminology and notation.

Take a positive power-bounded operator $T$ on a Banach lattice $E$ and assume that $E$ has a Banach lattice $F$ as predual on which $T$ has a positive preadjoint. Then we can construct the Köhler semigroup $\K(T)$ as above and obtain the following result.

\begin{theorem}\label{posJDLG}
    Take a positive power-bounded operator $T$ with preadjoint on a dual Banach lattice $E$  (the preadjoint is assumed to be a positive operator on the predual Banach lattice). Then for any decomposition $E=\rev\oplus\aws$ as in Theorem \ref{rechtstop-JdLG} the following properties hold.
    \begin{enumerate}[(a)]
        \item The space $\rev$ is a Banach lattice under the induced order for some equivalent norm.
        \item The operator $T$ is a lattice-isomorphism on the space $\rev$. 
    \end{enumerate}
\end{theorem}
 \begin{proof}
     The positive cone in $E$ is weak-* closed, therefore all operators in $\K(T)$ are positive. Thus every projection $P$ is a positive projection and by \cite[Prop. III.11.5]{schaefer1974Banachlattices} the space $\rev=\im{P}$ is a Banach lattice under the induced order and for an equivalent norm. Statement $(b)$ follows since $T$ is positive and invertible on $\rev$ with positive inverse in $P\K(T)$, hence a lattice isomorphism on the space $\rev$.
 \end{proof}

We now apply this information to the spectral theory of positive operators. The classical results in the finite dimensional case are due to O. Perron and G. Frobenius and have been generalized to positive operators on Banach lattices by many authors. We refer to \cite[Chap V.4]{schaefer1974Banachlattices} for a systematic treatment. In particular the peripheral spectrum of positive operators fulfills various symmetry properties. 
However the question whether every positive operator on every Banach lattice has cyclic peripheral spectrum is still open. See \cite{lotz1968spektrum} and \cite{gluck2016peripheral} for the terminology and the most general results about this cyclicity problem. 

Using the Köhler semigroup of the given positive operator, we are able to prove cyclicity in a special, but fundamental case. We expect that this technique will lead to new insights and results. 
\begin{theorem}\label{cyc}
    The unimodular spectrum of any power-bounded positive operator $T$ on a Banach lattice $E$ is cyclic.
\end{theorem}

\begin{proof}
    In a first step we pass to an ultraproduct $\hat{E}$ (see \cite[Chap V.1]{schaefer1974Banachlattices}), then we embed $\hat{E}$ into its bidual $(\hat{E})''$. After this we take a projection in the Köhler semigroup of the induced operator $\hat{T}''$ on this space and restrict onto the corresponding reversible part $((\hat{E})'')_{\mathrm{rev}}$, as visualized in the following diagram.  
    \[\begin{tikzcd}
E & {\hat{E}} & {(\hat{E})''} 
&\Tilde{E}\coloneqq {((\hat{E})'')_{\mathrm{rev}}}\\
T&\hat T &\hat{T}''&\Tilde T\coloneqq (\hat{T}'')|_{\rev}
	\arrow["\mathrm{ultraprod.}", hook, from=1-1, to=1-2]
	\arrow["\mathrm{bidual}",hook, from=1-2, to=1-3]
	\arrow["{\mathrm{restr.}}", twoheadrightarrow, from=1-3, to=1-4]
 \arrow[ from=2-1, to=2-2]
	\arrow[ from=2-2, to=2-3]
	\arrow[ from=2-3, to=2-4]
\end{tikzcd}\]
We denote the space $((\hat{E})'')_{\mathrm{rev}}$ by $\Tilde E$ and the restriction of the operator $\hat{T}''$ onto $\Tilde{E}$  by $\Tilde{T}$.
    
    For the induced operator of $T$ on $\hat{E}$ the unimodular spectrum (which we denote by $\per(T)$) consists of point spectrum $\sigma_{\mathrm{pnt}}(T)$ only, see \cite[Thm V.1.4]{schaefer1974Banachlattices}. Since there is a canonical embedding into the bidual space and the spectrum is invariant under taking adjoint operators, the unimodular spectrum of the induced operator on $(\hat{E})''$ is point spectrum as well. By Theorem \ref{rechtstop-JdLG} the eigenfunctions to unimodular eigenvalues are contained in $((\hat{E})'')_{\mathrm{rev}}$ and we conclude that \[\per(T)=\per(\hat{T}'')\sub \sigma_{\mathrm{pnt}}(\Tilde T).\] 
    Since $\Tilde T$ is a restriction of $(\hat T)''$ onto a closed subspace, we conclude that the approximate point spectrum satisfies $\sigma_{\mathrm{app}}(\Tilde{T})\sub\sigma(T)$ (see e.g. \cite[Lem 4.2]{grobler1995spectral}) and thereby \[\per(T)=\per(\Tilde{T})=\sigma_{\mathrm{pnt}}(\Tilde{T})\cap\sigma_{\mathrm{ per}}(\Tilde{T}).\] 
    
So without loss of generality it suffices to show that for any positive lattice isomophism $\Tilde T$ on a Banach lattice $\Tilde{E}$ the peripheral point spectrum is cyclic. Thus take any eigenvector $x$ to an unimodular eigenvalue $\lambda$. Then the ideal $\Tilde{E}_{|x|}$ is isomorphic to a space $C(K)$ of continuous functions on a compact Hausdorff space $K$ by the Theorem of Gelfand-Kakutani (see \cite[Chap. II.7]{schaefer1974Banachlattices}). So we extend the above diagram as follows. 
   \[\begin{tikzcd}
E & {\hat{E}} & {(\hat{E})''} 
&\Tilde{E}\coloneqq {((\hat{E})'')_{\mathrm{rev}}}&\Tilde{E}_{|x|}& C(K)
	\arrow["\mathrm{ultraprod.}", hook, from=1-1, to=1-2]
	\arrow["\mathrm{bidual}",hook, from=1-2, to=1-3]
	\arrow["{\mathrm{restrict}}", twoheadrightarrow, from=1-3, to=1-4]\arrow[hook,from=1-5,to=1-4]\arrow[leftrightarrow,from=1-5, to=1-6]
\end{tikzcd}\]
Since  $\Tilde T|x|=|\Tilde{T}x|=|x|$, the ideal $\Tilde{E}_{|x|}$ is invariant under $T$, so $\Tilde T$ becomes a positive Markov lattice homomorphism on $C(K)$. Since those operators are precisely the Markov algebra homomorphisms for the multiplication in $C(K)$ (see \cite[III.9.1]{schaefer1974Banachlattices}), we obtain
    \[\Tilde{T}x^k=(\Tilde{T}x)^k=\lambda^k x^k,\]
    and hence $\lambda^k\in\per(T)$, so the peripheral spectrum is cyclic.
\end{proof}
The preceeding proof even shows the following interesting fact.
 \begin{corollary}\label{latics}
To every positive power-bounded operator $T$ with $r(T)=1$ there corresponds a lattice isomophism $\Tilde T$ on another Banach lattice such that \[\per(T)=\per(\Tilde{T})\sub\sigma_{\mathrm{pnt}}(\Tilde{T}).\] 
\end{corollary}
\begin{remark}
   The above arguments also work for various generalizations. 
    \begin{enumerate}[(i)]
        \item 
  For power bounded uniformly asymptotically positive operators \newline(i.e. $d(T^nx,E_+)\to 0$ uniformly), see \cite{Glueck2017}, all operators in $\K(T)$ corresponding to $\beta\N\setminus\N$ are positive, hence also all projections in $\K(T)$ and the induced operators on $\rev$ are positive. Thus one obtains the statements of Theorems \ref{posJDLG} and \ref{cyc} in the case of power bounded uniformly asymptotically positive operators as well. Theorem \ref{posJDLG} only needs individually  asymptotical positivity, the uniform convergence is only needed when taking the ultraproduct.
  
    \item The arguments work similarly for other discrete semigroups $S$ instead of $\N$. If one takes the semigroup to be commutative, then an analogue of Theorem \ref{rechtstop-JdLG} holds if $\K(T)$ is replaced by the pointwise weak-* closure of a bounded representation of $S$ on $E$, and the projection is chosen to be a minimal idempotent in this closure.
    \item A generalization to non power-bounded operators is possible, e.g., for Abel-bounded positive operators. Here one assumes that there exists a positive projection onto the fixed space of $T$ in the (uniformly) closed convex hull of $\{T^n:\,n\in\N\}$. Since this projection dominates $T$ on the space of all superfixed vectors (i.e. $x\ge 0$ such that $Tx\ge x$), one obtains power-boundedness of the operator $T$ restricted to some principal ideal and then proceeds with the power-bounded theory. Similarly, one can prove the cyclicity result for positive (WS)-solvable operators from  \cite{gluck2016peripheral}. This in turn can be combined with the generalizations from part $(i)$.
      \end{enumerate}
\end{remark}
\end{document}